\documentclass[11pt]{article}
\usepackage{graphicx}
\usepackage{amsthm, amsmath, amssymb, tikz, bm}
\usepackage{mathtools}
\usepackage{xifthen}

\usepackage[margin=1.2in]{geometry} 

\usepackage[shortlabels]{enumitem}
\usepackage{todonotes}
\usepackage[colorlinks=true,
linkcolor=blue,citecolor=blue,
urlcolor=blue]{hyperref}

\usetikzlibrary{calc,shapes, backgrounds}

\newtheorem{lemma}{Lemma}
\newtheorem{theorem}{Theorem}

\newtheorem{remark}{Remark}

\newcommand{\ds}{\displaystyle}

\newcommand{\lp}{\left (}
\newcommand{\rp}{\right )}

\newcommand{\cH}{\mathcal{H}}

\newcommand{\br}{\hat{R}}

\DeclarePairedDelimiter{\abs}{\lvert}{\rvert}%

\newcommand{\Prob}[1]{\Pr\lp#1\rp}

\title{On the cover Ramsey number of Berge hypergraphs}

\author{
Linyuan Lu
\thanks{University of South Carolina, Columbia, SC 29208,
({\tt lu@math.sc.edu}). This author was supported in part by NSF grant DMS-1600811.} \and
Zhiyu Wang \thanks{University of South Carolina, Columbia, SC 29208,
({\tt zhiyuw@math.sc.edu}).} 
}

\begin{document}

\maketitle

\begin{abstract}
For a fixed set of positive integers $R$, we say $\cH$ is an $R$-uniform hypergraph, or $R$-graph, if the cardinality of each edge belongs to $R$. An $R$-graph $\cH$ is \emph{covering} if every vertex pair of $\cH$ is contained in some hyperedge. For a graph $G=(V,E)$, a hypergraph $\cH$ is called a \textit{Berge}-$G$, denoted by $BG$, if there exists an injection $f: E(G) \to E(\cH)$ such that for every $e \in E(G)$, $e \subseteq f(e)$. In this note, we define a new type of Ramsey number, namely the \emph{cover Ramsey number}, denoted as $\hat{R}^R(BG_1, BG_2)$, as the smallest integer $n_0$ such that for every covering $R$-uniform hypergraph $\mathcal{H}$ on $n \geq n_0$ vertices and every $2$-edge-coloring (blue and red) of $\mathcal{H}$ , there is either a blue Berge-$G_1$ or a red Berge-$G_2$ subhypergraph. We show that for every $k\geq 2$, there exists some $c_k$ such that for any finite graphs $G_1$ and $G_2$, $R(G_1, G_2) \leq \hat{R}^{[k]}(BG_1, BG_2) \leq c_k \cdot R(G_1, G_2)^3$. Moreover, we show that for each positive integer $d$ and $k$, there exists a constant $c = c(d,k)$ such that if $G$ is a graph on $n$ vertices with maximum degree at most $d$, then $\hat{R}^{[k]}(BG,BG) \leq cn$.
\end{abstract}

\section{Introduction}
A hypergraph is a pair $\cH=(V,E)$ where $V$ is a vertex set and $E\subseteq 2^V$ is an edge set. For a fixed set of positive integers $R$, we say $\cH$ is an $R$-uniform hypergraph, or $R$-graph for short, if the cardinality of each edge belongs to $R$. 
If $R=\{k\}$, then an $R$-graph is simply a $k$-uniform hypergraph or a $k$-graph.
Given an $R$-graph $\cH = (V,E)$ and a set $S \in \binom{V}{s}$, let deg$(S)$ denote the number of edges containing $S$ and $\delta_s(\cH)$ be the minimum $s$-degree of $\cH$, i.e., the minimum of deg$(S)$ over all $s$-element sets $S \in \binom{V}{s}$. When $s = 2$, $\delta_2(\cH)$ is also called the minimum \textit{co-degree} of $\cH$. Given a hypergraph $\cH$, the \textit{$2$-shadow}(or \textit{shadow}) of $\cH$, denoted by $\partial_2(\cH)$, is a simple $2$-uniform graph $G=(V,E)$ such that $V(G)=V(\cH)$ and $uv\in E(G)$ if and only if $\{u,v\}\subseteq h$ for some $h\in E(\cH)$. Note that $\delta_2(\cH) \geq 1$ if and only if $\partial_2(\cH)$ is a complete graph. In this case, we say $\cH$ is {\em covering}.

There are several notions of a path or a cycle in hypergraphs. A \textit{Berge path} of length $t$ is a collection of $t$ hyperedges $h_1, h_2, \ldots, h_{t} \in E$ and $t+1$ vertices $v_1, \ldots, v_{t+1}$ such that $\{v_i, v_{i+1}\} \subseteq h_i$ for each $i\in [t]$. Similarly, a $k$-graph $\cH = (V,E)$ is called a \textit{Berge} cycle of length $t$ if $E$ consists of $t$ distinct edges $h_1, h_2, \ldots, h_t$ and $V$ contains $t$ distinct vertices $v_1, v_2, \ldots, v_t$ such that $\{v_i, v_{i+1}\} \subseteq h_i$ for every $i\in [t]$ where $v_{t+1} \equiv v_1$. Note that there may be other vertices than $v_1, \ldots, v_t$ in the edges of a Berge cycle or path. Gerbner and Palmer \cite{GP} extended the definition of Berge paths and Berge cycles to general graphs. In particular, given  a simple graph $G$, a hyperge graph $\cH$ is called  \emph{Berge-$G$} if there is a bijection $f:E(G) \to E(\cH)$ such that for all $e \in E(G)$, we have $e \subseteq f(e)$.

We say an $R$-graph $\cH$ on $n$ vertices contains a \textit{Hamiltonian Berge cycle (path)} if it contains a Berge cycle (path) of length $n$ (or $n-1$). We say $\cH$ is {\em Berge-Hamiltonian} if it contains a Hamiltonian Berge cycle. Bermond, Germa, Heydemann, and Sotteau \cite{BGHS} showed a Dirac-type theorem for Berge cycles. We showed in \cite{LW-codegree} that for every finite set $R$ of positive integers, there is an integer $n_0=n_0(R)$ such that every covering $R$-uniform hypergraph $\cH$ on $n$ ($n\geq n_0$) vertices contains a Berge cycle $C_s$ for any $3\leq s\leq n$. In particular, every covering $R$-graph on sufficiently large $n$ vertices is Berge-Hamiltonian.

Extremal problems related to Berge hypergraphs have been receiving increasing attention lately. For Turan-type results, Gy\H{o}ri, Katona and Lemons \cite{GKL} showed that for a $k$-graph $\cH$ containing no Berge path of length $t$, if $t\geq k+2 \geq 5$, then $e(\cH) \leq \frac{n}{t}\binom{t}{k}$; if $3\leq t\leq k$, then $e(\cH) \leq \frac{n(t-1)}{k+1}$.
The remaining case of $t=k+1$ was settled by Davoodi, Gy\H{o}ri, Methuku and Tompkins \cite{Davoodi}.
For long cycles, F\"{u}redi, Kostochka and Luo \cite{FKL} showed that for $k\geq 3$ and $t\geq k+3$, if $\cH$ is an $n$-vertex $k$-graph with no Berge cycle of length at least $t$, then $e(\cH) \leq \frac{n-1}{t-2}\binom{t-1}{k}$. The equality is achieved if and only if $\partial_2(\cH)$ is connected and for every block $D$ of $\partial_2(\cH)$, $D = K_{t-1}$ and $\cH[D] = K_{t-1}^k$. The cases for $t\in \{k+1, k+2\}$ are settled by Ergemlidze et al. \cite{Ergemlidze}. The case when $t= k$ is recently settled by Gy\H{o}ri et al \cite{GLSZ}. For general results on the maximum size of a Berge-$G$-free hypergraph for an arbitrary graph $G$, see for example \cite{GMP, GMT, PTTW}.

For Ramsey-type results, define $R_c^k(BG_1, \ldots, BG_c)$ as the smallest integer $n$ such that  for any $c$-edge-coloring of a complete $k$-uniform hypergraph on $n$ vertices, there exists a Berge-$G_i$ subhypergraph with color $i$ for some $i$. Salia, Tompkins, Wang and Zamora \cite{STWZ} showed that $R_2^3(BK_s, BK_t) = s+t -3$ for $s,t \geq 4$ and $\max(s,t) \geq 5$. For higher uniformity, they showed that $R^4(BK_t, BK_t) = t+1$ for $t \geq 6$ and $R_2^k(BK_t, BK_t)=t$ for $k\geq 5$ and $t$ sufficiently large.
Independently and more generally, Gerbner, Methuku, Omidi and Vizer \cite{GMOV} showed that $R_c^k(BK_n) = n$ if $k > 2c$; $R_c^k(BK_n) = n+1$ if $k = 2c$ and obtained bounded on $R_c^k(BK_n)$ when $k < 2c$. They also determined the exact value of $R_2^3(BT_1, BT_2)$ for every pair of trees. Similar investigations have also been started independently by Axenovich and Gy\'arf\'as \cite{AG} who focus on the Ramsey number of small fixed graphs where the number of colors may go to infinity.

Although it is pleasant to see that the Ramsey number of Berge cliques is linear when the number of colors are not too big relative to the uniformity, the result is also not surprising due to the fact that $K^k_n$ has much more edges than $B K_t$. This motivates us to define a new type of Ramsey number such that the host graph has relatively small number of edges. In particular, inspired by the results in \cite{LW-codegree}, we realize that the covering property of a hypergraph is closely related to finding Berge subhypergraphs. Hence we define a new type of Ramsey number, namely \textit{cover Ramsey number}, denoted as $\br^R(BG_1, BG_2)$, as the smallest integer $n_0$ such that for every covering $R$-uniform hypergraph $\cH$ on $n \geq n_0$ vertices and every $2$-edge-coloring of $\cH$ with blue and red, there is either a blue Berge-$G_1$ or a red Berge-$G_2$ subhypergraph in $\cH$. Note that when $R = \{2\}$,  $\br^R(BG_1, BG_2)$ is exactly the classical Ramsey number. For ease of reference, we use $\br^{k}(BG_1, BG_2)$ to denote $\br^{\{k\}}(BG_1, BG_2)$. It is easy to see that 
$\br^{k}(BG_1, BG_2) \leq \br^{[k]}(BG_1, BG_2).$

Let $R_c(G_1, \ldots, G_c)$ denote the classical multi-color Ramsey number, i.e., the smallest integer $n$ such that any $c$-edge-coloring of $K_n$ contains a monochromatic $G_i$ in the $i$-th color for some $i\in [c]$. When $c = 2$, we simply write $R_2(G_1, G_2)$ as $R(G_1, G_2)$. We first show the following theorem.

\begin{theorem}\label{thm:Ramsey-bound}
For every $k\geq 2$, there exists some constant $c_k$ such that for any two non-empty finite graphs $G_1$ and $G_2$,
\[R(G_1, G_2) \leq \br^{[k]}(BG_1, BG_2) \leq c_k \cdot R(G_1, G_2)^3.\]
\end{theorem}

Theorem \ref{thm:Ramsey-bound} implies $\br^{R}(BG_1, BG_2)$ is always finite, thus well-defined. In fact, let $k$ be the greatest integer in $R$. We have $R\subseteq [k]$ and
$$ \br^{R}(BG_1, BG_2)\leq  \br^{[k]}(BG_1, BG_2)\leq c_k \cdot R(G_1, G_2)^3.$$

Note that Theorem \ref{thm:Ramsey-bound} doesn't give a lower bound for $\br^k(BG_1, BG_2)$.
For complete graphs $K_t$, we show that the cover Ramsey number of Berge cliques is at least exponential in $t$. Note that this is very different from the hypergraph Ramsey number of Berge cliques (see \cite{STWZ} and \cite{GMOV}), which is linear.

\begin{theorem}\label{thm:exp-lower}
For every $k \geq 2$ and sufficiently large $t$, we have that 
 \[ \br^{k}(BK_t, BK_t) > (1+o(1))\frac{\sqrt{2}}{e} t2^{t/2}.\]
\end{theorem}

\begin{remark}
For a fixed $t$ and $R \subseteq [k]$, let $N(t)$ be the set of integers $n$ such that for every covering $R$-uniform hypergraph $\cH$ on $n$ vertices and every $2$-edge-coloring of $\cH$, there is a monochromatic Berge-$K_t$. We remark that $N(t)$ may not be a single interval. However, by Theorem \ref{thm:Ramsey-bound}, there exists some $n_0$ such that $[n_0, \infty) \subseteq N(t)$.
\end{remark}

For a graph $G$ with bounded maximum degree, Chv\'atal, R\"{o}dl, Szemer\'edi and Trotter showed in \cite{CRST} that for each positive integer $d$, there exists a constant $c = c(d)$ such that if $G$ is a graph on $n$ vertices with $\Delta(G)\leq d$, then $R(G,G) \leq cn$. In this note, we show that the cover Ramsey number of Berge bounded-degree graphs is also linear. The proof uses a modification of the proof of Chv\'atal, R\"{o}dl, Szemer\'edi and Trotter in \cite{CRST} that allows for more than two colors. 

\begin{theorem}\label{thm:bounded-degree}
For each positive integer $d$ and $k$, there exists a constant $c = c(d,k)$ such that if $G$ is a graph on $n$ vertices with maximum degree at most $d$, then 
\[\br^{[k]}(BG,BG) \leq cn.\]
\end{theorem}

Theorem \ref{thm:bounded-degree} implies that for fixed integers $k$ and $d$, there is a constant $c:=c(d,k)$ such that $\br^{[k]}(BG,BG)\leq c R(G,G)$ holds for any graph $G$ with maximum degree at most $d$.
It is an interesting question whether
$\ds\lim_{t\to\infty}\frac{\br^{[k]}(BK_t,BK_t)}{R(K_t,K_t)}=\infty$
for all $k\geq 3$.

\section{Proof of Theorem \ref{thm:Ramsey-bound}}
\begin{proof}[Proof of Theorem \ref{thm:Ramsey-bound}]
The lower bound that $\br^{[k]}(BG_1, BG_2) \geq R(G_1, G_2)$ is clear from the definition since $R(G_1, G_2) =  \br^{\{2\}}(BG_1, BG_2) \leq  \br^{[k]}(BG_1, BG_2)$.

For the upper bound, given $k \geq 2$, set $c_k = k^3/12$. Let $\cH=(V,E)$ be a $2$-edge-colored $R$-graph on $n = c_k R(G_1, G_2)^3$ vertices.
Assume further that $\cH$ is edge-minimal with respect to the covering property. Suppose $E = \{h_1, h_2, \ldots, h_m\}$ where $m = |E|$. Since $\cH$ is edge-minimal and covering, it follows that 
$\binom{n}{2}/\binom{k}{2} \leq m \leq \binom{n}{2}$.

Now let $S\subseteq V$ be a uniformly and randomly chosen subset of $V$ of size $s = R(G_1, G_2)$. For each $i\in [m]$, let $B_i$ be the event that $|h_i \cap S| \geq 3$. It is not hard to see that \[
    \Prob{B_i} \leq \binom{k}{3}\frac{\binom{n-3}{s-3}}{\binom{n}{s}}. 
\]
Taking a union bound over all $B_i$, we have that \begin{align*}
    \Prob{B_1 \vee \ldots \vee B_m} &\leq \binom{n}{2}\binom{k}{3}\frac{\binom{n-3}{s-3}}{\binom{n}{s}}\\
    &= 3\frac{\binom{k}{3}\binom{s}{3}}{n-2}\\
    &< 1.
\end{align*}
The last step is due to the following inequality:
$$n=\frac{k^3}{12}s^3
\geq 3\lp\binom{k}{3}+1\rp
\lp\binom{s}{3}+1\rp > 3+ 3\binom{k}{3} \binom{s}{3}
$$
for any $k\geq 2$ and $s\geq 2$.
 Hence with positive probability, there exists $S \subseteq V$ with $|S| = R(G_1,G_2)$ such that every hyperedge intersects $S$ in at most $2$ points. Now consider the trace of $\cH$ on $S$, denoted by $G = \cH_S$. By the covering property and the choice of $S$, $G$ is a complete  graph (ignoring edges of cardinality $1$). Recall that by definition, $E(G) = \{h \cap S: h \in E(\cH)\}$. Hence for each edge $e \in G$, there exists some $h = \phi(e) \in E(\cH)$ such that $e = h\cap S$. Moreover, for $e_1\neq e_2$, $\phi(e_1) \neq \phi(e_2)$ due to the choice of $S$. Now for each edge $e\in E(G)$, color the edge $e$ with the same color of $\phi(e)$ in $\cH$. Since $|S| = R(G_1, G_2)$, it follows that there exists either a blue $G_1$ or a red $G_2$ in $G$, which corresponds to a blue Berge $G_1$ or a red Berge $G_2$ in $\cH$. This shows that $\br^{[k]}(BG_1, BG_2) \leq k^3/12 \cdot R(G_1, G_2)^3.$
\end{proof}

\section{Proof of Theorem \ref{thm:exp-lower}}

The construction comes from a random $2$-edge-coloring of a covering $k$-unifrom hypergraph that is obtained from a combinatorial design. 

A {\it resolvable} BIBD, denoted as ${\rm BIBD}(n, k,\lambda)$, is a collection $P_1,\ldots,P_m$ of partitions of an underlying $n$-element set into  $k$-element subsets such that every $2$-element subset of the $n$-element set is contained by exactly $\lambda$ of the  $\frac{mn}{k}$ $k$-element sets listed in the partitions. We restrict ourselves to $\lambda=1$, that is, each $2$-element subset of the $n$-element set is contained in one and only one of the $k$-element sets listed in the partitions. 
 
Note that the existence of such a design implies that $|P_i|=\frac{n}{k}$ and $m\frac{n}{k}\binom{k}{2}=\binom{n}{2}$, i.e. $m=\frac{n-1}{k-1}$, which gives the well known necessary condition that $n\equiv k \pmod{k(k-1)}$ for the existence of such a resolvable BIBD. For the $k=3$ case (which is commonly called a Kirkman triple system to honor Kirkman~\cite{kirkman} who posed the problem) it is also a sufficient condition \cite{kirk}, and for $k=4$ the corresponding $n\equiv 4 \pmod{12}$ is also a sufficient condition \cite{hanani}. For every $k$, the congruence is also a sufficient condition for all $n>n_0(k)$ \cite{RV}. Also, for every even $k\geq 4$, the congruence implies existence for $n>\exp\{\exp\{k^{18k^2}\}\}$ \cite{chang}. 

\begin{proof}[Proof of Theorem \ref{thm:exp-lower}]
For a fixed $k\geq 2$, let $t_0$ be sufficiently large such that for all $n\geq (1+o(1))\frac{t_0\sqrt{2}}{e}2^{t_0/2}$ and $n\equiv k \pmod{k(k-1)}$, a resolvable BIBD $(n,k,1)$ exists.

Let $t\geq t_0$ and $n=(1+o(1))\frac{t\sqrt{2}}{e} 2^{t/2}$. Assume that $n$ is an integer such that a resolvable BIBD $(n,k,1)$ exists. Let $\cH = (V,E)$ be a $k$-uniform hypergraph such that $V$ is the underlying $n$-element set of the resolvable BIBD $(n,k,1)$ and $E$ is the collection of $k$-element sets listed in the partitions $P_1, \ldots, P_m$. Note that by the definition of $(n,k,1)$, $\cH$ is a covering $k$-graph with $\binom{n}{2}/\binom{k}{2}$ edges and every vertex pair of $\cH$ is contained in exactly one hyperedge. 

 Our goal is to construct a coloring of $\cH$ with no monochromatic $BK_t$ as subhypergraph. Color each hyperedge of $\cH$ in blue and red uniformly and randomly with probability $1/2$.
For any set $S$ of $t$ vertices, let $A_S$ be the bad event that $S$ induces a monochromatic $BK_t$. We will apply the Lovasz Local Lemma to show that we can avoid all bad events $\{A_S\colon S\subseteq V \mbox{ and } |S|=t\}$.

Note that by the definition of $(n,k,1)$, for each vertex pair of $S$, there exists a unique hyperedge containing that vertex pair. Hence there is at most one Berge-$K_t$ with $S$ as the underlying vertex set. Furthermore, if there is a Berge-$K_t$ with $S$ as the underlying vertex set, then the hyperedges containing the vertex pairs of $S$ are all distinct. 
Hence 
\[\Prob{A_S }=\left\{\begin{array}[c]{ll}
2^{1-\binom{t}{2}} &\mbox{ if there is no $h \in E(\cH)$ such that $|h\cap S|\geq 3$} ,\\
0 & \mbox{ otherwise}.
\end{array}\right.
\]
Two bad events $A_S$ and $A_T$ are independent if there is no edge $f$ intersecting both $S$ and $T$ on exactly two vertices. For a fixed event $A_S$, the number $d$ of bad events $A_T$ dependent  on $A_S$ satisfies
$$d\leq \binom{t}{2}\binom{k}{2}\binom{n-2}{t-2}-1.$$
Applying the symmetric version of the Lovasz Local Lemma \cite{Spencer}, if
$e(d+1)\Prob{A_S}<1$ for all $S$, then
$\Prob{\bigwedge_{S}\overline{A_S}}>0$.

It suffices to have 
\[e\cdot \binom{t}{2}\binom{k}{2}\binom{n-2}{t-2}2^{1-\binom{t}{2}}<1,
\]
which is satisfied if we choose 
$n=(1+o(1))\frac{\sqrt{2}}{e} t 2^{t/2}$.
Hence there exists a coloring of $\cH$ with no monochromatic Berge $K_t$ as subhypergraph. It follows by definition that $\br^k(BK_t, BK_t)> (1+o(1))\frac{\sqrt{2}}{e} t 2^{t/2}$.

\end{proof}

\section{Proof of Theorem \ref{thm:bounded-degree}}

The proof of Theorem \ref{thm:bounded-degree} uses a modification of the proof Chv\'atal, R\"{o}dl, Szemer\'edi and Trotter in \cite{CRST} to allow for more than two colors. For the reason of self-completeness, we state and give the details in this section. Let $R_c(G)$ denote the multicolor Ramsey number $R_c(G, G, \ldots, G)$.

\begin{theorem}\cite{CRST}\label{thm:CRST}
For each positive integer $c$ and $d$, there exists a constant $C = C(c,d)$ such that if $G$ is a graph on $n$ vertices with maximum degree at most $d$, then $R_c(G) \leq Cn$.
\end{theorem}

We first show how Theorem \ref{thm:CRST} implies Theorem \ref{thm:bounded-degree}.

\begin{proof}[Proof of Theorem \ref{thm:bounded-degree}]
For fixed positive integers $d$ and $k$, let $C = C(2\binom{k}{2},d)$ be the constant obtained from Theorem \ref{thm:CRST}. We will show that if $G$ is a graph on $n$ vertices with maximum degree at most $d$, then $\br^{[k]}(BG,BG) \leq Cn$.

Let $\cH = (V,E)$ be a $2$-edged-colored covering $[k]$-graph on $N = Cn$ vertices. Suppose $E = \{h_1, \ldots, h_m\}$. For each $h_i$, give each vertex pair $uv \subseteq h_i$ a unique label $\phi_{h_i}(uv)$ in $[\binom{k}{2}]$. Now consider a $2\binom{k}{2}$-edge coloring of $K_N$: for each $uv \in E(K_N)$, pick an arbitrary hyperedge $h \in E(\cH)$ such that $\{u,v\} \subseteq h$. Such $h$ exists since $\cH$ is covering. If $h$ is colored blue in $\cH$, then color $uv \in E(K_N)$ with a color represented by the ordered pair $(1, \phi_h(uv))$; if $h$ is red in $\cH$, then color $uv$ with a color represented by the ordered pair $(2, \phi_h(uv))$. Note that $K_N$ is a $2\binom{k}{2}$-edge-colored graph. Since $N = Cn$, by the definition of multi-color Ramsey number, it follows that if $G$ is a graph on $n$ vertices with maximum degree at most $d$, then $K_N$ contains a monochromatic $G$ as subgraph. WLOG, suppose $G$ is colored $(1, r)$ where $1\leq r \leq \binom{k}{2}$. Now by our construction, for each $e \in E(G)$, there exists hyperedge $h = h(e)$ such that $h$ is colored blue in $\cH$. Moreover we claim that for $e_1 \neq e_2 \in E(G)$, $h(e_1) \neq h(e_2)$. Suppose not, i.e., $h$ contains both $e_1$ and $e_2$. Then $\phi_h(e_1) \neq \phi_h(e_2)$, which contradicts that $e_1, e_2$ receives the same color in $K_N$. Hence, it follows that we can find a monochromatic Berge copy of $G$ in $\cH$.
\end{proof}

In the remaining of this section, we will give a proof of Theorem \ref{thm:CRST}. We remark again that the proof follows along the same line of \cite{CRST} and we are only giving the details here for the sake of self-completeness. 

As suggested by \cite{CRST}, the proof requires a generalization of the regularity lemma, which is an easy modification of the original proof in \cite{Szemeredi}. Given a graph $G$, let $V(G) = A_1 \cup A_2 \cup \cdots \cup A_k$ be a partition of $V(G)$ into disjoint subsets. We call such partition \textit{equipartite} if $\abs{|V_i|-|V_j|} \leq 1$ for all $i,j \in [k]$. Moreover, given two disjoint sets $X,Y \subseteq V(G)$, the \textit{edge density} of $(X,Y)$, denoted as $d(X,Y)$, is defined as $d(X,Y) = |e(X,Y)|/|X||Y|$ where $e(X,Y) = \{xy \in E(G): x \in X, y\in Y\}$. 

\begin{lemma}\label{lem:reg}
For every $\epsilon > 0$ and integers $c, m$, there exists an $M$ and $N_0$ such that if the edges of a graph $G$ on $n\geq N_0$ vertices are $c$-colored, then there exists an equipartite partition $V(G) = A_1 \cup A_2 \cup \ldots \cup A_k$ for some $m\leq k \leq M$, such that all but at most $\epsilon k^2$ pairs $(A_i, A_j)$ are $\epsilon$-regular: for every $X \subseteq A_i$ and $Y \subseteq A_j$ with $|X| \geq \epsilon|A_i|$, $|Y| \geq \epsilon|A_j|$, we have 
\[ \abs{d_s(X,Y) - d_s(X,Y)} < \epsilon\]
for each $s \in [c]$ where $d_s$ is the edge-density in the $s$-th color.
\end{lemma}

\begin{proof}[Proof of Theorem \ref{thm:CRST}]

Let $d$ be any positive integer. Let $N$ be large enough so that if we define $\epsilon = 1/N$, then   $\frac{1}{c\log (2c)} \log\lp \frac{1}{2\epsilon} \rp \geq d+1$. Observe that with this choice of $N$, we also have $1/(2c)^d > 2d^2 \epsilon$. Let $M, N_0$ be the constants given by Lemma \ref{lem:reg} when $c$ is the number of colors and $m = 1/\epsilon$. Set $C = C(c,d) = \max\{N_0, M/d^2\epsilon\}$.

Now let $G$ be a graph on $n$ vertices $x_1, \ldots, x_n$ and maximum degree at most $d$. Consider an arbitrary $c$-coloring of $K_{Cn}$. 
Let $H_1, \ldots, H_c$ denote the subgraphs of $G$ induced by each of the $c$ colors respectively. 
By Lemma \ref{lem:reg}, there exists an equipartite partition $V(K_{Cn}) = A_1 \cup A_2 \cup \ldots \cup A_k$ that satisfies the regularity condition for each color class, i.e., for each $i \in [c]$, $V(H_i) =  A_1 \cup A_2 \cup \ldots \cup A_k$ gives an equipartite $\epsilon$-regular partition.

Let $H^*$ denote the graph whose vertex set is $\{A_i: i\in [k]\}$ and $A_i A_j$ is an edge if and only $(A_i, A_j)$ is $\epsilon$-regular in $H$. By Lemma \ref{lem:reg}, $|E(H)| \geq (1-\epsilon)\binom{k}{2}$. Hence by Turan's theorem, there exists a complete subgraph $H^{**}$ of $H^*$ of size at least $1/2\epsilon$. WLOG (with relabeling), assume that $V(H^{**}) = \{A_i: 1\leq i\leq 1/2\epsilon\}$. Now for each $A_i ,A_j \in V(H^{**})$, color the edge $A_iA_j$ with color $s$ if $d_s(A_i,A_j)$ is the largest among all colors in $[c]$ (break arbitrarily if the same). 
Recall that $R_c(K_{t}) \leq c^{ct}$ and $\frac{1}{c\log (2c)} \log\lp \frac{1}{2\epsilon} \rp \geq d+1$ by our assumption. Hence we have that $1/2\epsilon \geq R_c(K_{d+1})$. Then it follows from Ramsey's theorem that there is a monochromatic complete subgraph $H^{***}$ with $d+1$ vertices. WLOG, $H^{***}$ is in color $1$. Then we can relabel the sets in the partition so that 
\begin{enumerate}[(i)]
    \item $(A_i, A_j)$ is $\epsilon$-regular, and 
    \item $d_1(A_i, A_j) \geq \frac{1}{c}$
\end{enumerate}
for all $i,j$ with $1\leq i< j\leq d+1$. We then claim that $H_1$ contains a copy of $G$. Recall that $V(G) = \{x_i: i\in [n]\}$. We will choose $y_1, y_2,\ldots, y_n \in V(H_1)$ inductively so that the map $\phi: x_i \to y_i$ is an embedding of $G$ in $H_1$.
In particular, the points are chosen so that for each $i\in [n]$, the following are satisfied:
\begin{enumerate}[(a)]
    \item $y_t \in A_j$ for some $j \in [d+1]$ for each $t\in [i]$.
    \item For $t_1, t_2 \in [i]$, if $x_{t_1}x_{t_2}\in E(G)$, then $y_{t_1},y_{t_2}$ are adjacent in $H_1$ and are in different partition. 
    \item For $i < t \leq n$, define $V(t,i) = \{y_j: j\in [i], x_j x_t \in E(G)\}$. For each $r \in [d+1]$ such that $A_r \cap V(t,i) = \emptyset$, $A_r$ contains a subset $A'_r$ having at least  $|A'_r|/(2c)^{|V(t,i)|}$ so that every point in $A'_r$ is adjacent to every point in $V(t,i)$.
\end{enumerate}

Suppose that for some $i \in [n]$, the points $\{y_t:  t\leq [i]\}$ are already chosen so that the conditions $(a)$-$(c)$ above are satisfied.
We will then pick $y_{i+1}$ so that conditions $(a)$-$(c)$ remain true.

First pick some $r_0 \in [d+1]$ so that $A_{r_0} \cap V(i+1,i) =\emptyset$. This is possible since the degree of $x_{i+1}$ is at most $d$. By condition $(c)$, there exists $A'_{r_0} \subseteq A_{r_0}$ such that $\abs{A'_{r_0}} \geq \abs{A_{r_0}}/(2c)^{\ell}$ where $\ell = \abs{V(i+1,i)}$. Moreover, each vertex of $A'_{r_0}$ is adjacent to every vertex of $V(i+1,i)$. 
It's easy to see that with any choice of $y_{i+1}$ from $A'_{r_0}$, condition $(a)$ and $(b)$ are clearly satisfied.
For condition $(c)$, observe we only need to handle the values of $i+1 < t \leq n$ such that $x_t x_{i+1} \in E(G)$. There are at most $d$ such values since $d(x_{i+1}) \leq d$. Pick one such $t$ arbitrarily.  
Now pick an arbitrary $r \neq r_0$ such that $A_{r} \cap V(t,i) = \emptyset$. Observe $\ell' = |V(t,i+1)| = |V(t,i)| + 1$. By condition $(c)$, we already know that there exists some $A'_r \subseteq A_r$ such that $|A'_r|\geq |A_r|/(2c)^{\ell'-1} \geq \epsilon |A_r|$ and every vertex of $A'_r$ is adjacent to every vertex of $V(t,i)$. Now since $(A_{r}, A_{r_0})$ is $\epsilon$-regular and $d_1(A_r, A_{r_0})\geq \frac{1}{c}$, it follows that at most $\epsilon |A_{r_0}|$ of the points in $A'_{r_0}$ are adjacent to less than $\frac{1}{2c}$ of the points in $A'_r$. Fixing $t$ and proceeding through all values of $r$, we would eliminate at most $d \epsilon|A_{r_0}|$ candidates for $y_{i+1}$ in $A'_{r_0}$. Ranging over all of the $d$ possible values of $t$, we then eliminate at most $d^2 \epsilon |A_{r_0}|$ candidates of $y_{i+1}$ in $A'_{r_0}$. Moreover, there are at most $n$ points in $A'_{r_0}$ that may have been selected previously already. Since the number of partitions $k \leq M$ and $C \geq M/d^2\epsilon$, we have that $|A_{r_0}| \geq Cn/M$, which implies that $n \leq d^2 \epsilon |A_{r_0}|$.

In order to be able to pick $y_{i+1}$, it suffices to show that $|A'_{r_0}| > 2d^2\epsilon |A_{r_0}|$. This holds because $|A'_{r_0}|/|A_{r_0}| > 1/(2c)^d > 2d^2\epsilon$. This completes the proof of the theorem.

\end{proof}

\end{document}